\documentclass{amsart} 						
\usepackage{graphicx}						
\usepackage{amssymb}						
\usepackage{epstopdf}
 \usepackage{eucal}
\theoremstyle{plain}							
\newtheorem{Lem}{Lemma}[section]
\newtheorem{Thm}[Lem]{Theorem}
\newtheorem{Prop}[Lem]{Proposition}

\theoremstyle{definition}
\newtheorem*{Def}{Definition}
\newcommand{\set}{{\rm{SET}^\circledR}}
\usepackage{subfigure}
\usepackage[rflt]{floatflt}
\usepackage{wrapfig}
\usepackage{capt-of}
\usepackage{color}
\usepackage{fullpage}
\begin{document}

\title{Partitions of $AG(4,3)$ into Maximal Caps}
\author{{Michael Follett}
\and{Kyle Kalail}
\and{Elizabeth McMahon}
\and{Catherine Pelland}
\and{Robert Won}}
\address{Lafayette College\\
   Easton, PA  18042\\ {\tt mcmahone@lafayette.edu}}
 \thanks{Research  supported by NSF grant DMS-055282. The first author also had support from a Lafayette College EXCEL program grant}

\date{}


\begin{abstract}
In a geometry, a maximal cap is a collection of points of largest size containing no lines.
In $AG(4,3)$, maximal caps contain 20 points.  The 81 points of $AG(4,3)$ can be partitioned into 4 mutually disjoint maximal caps together with a single point $P$, where every pair of points that makes a line with $P$ lies entirely inside one of those  caps.  The caps in a partition can be paired up so that both pairs are either in exactly one partition or they are both in two different partitions.  This difference determines the two equivalence classes of partitions of $AG(4,3)$  under the action by affine transformations. 
\end{abstract}
\maketitle

\section{Introduction}

A {\em $k$-cap} (or briefly a {\em cap}) in $AG(n,3)$ is a set of $k$ points containing no 3 points on a line; a {\em maximal cap} is a cap of largest possible size.  A cap is called {\em complete} if it is not a subset of a larger cap.  There are caps in $AG(n,3), \, n \ge 3$, which are complete but smaller than a maximal cap.

The elements of $AG(n,3)$ can be written as $n$-tuples with coordinates in $\mathbb{Z}_3$, so the full transformation group of $AG(n,3)$  is the affine group $\mathit{Aff}(n,3) = GL(n,3) \ltimes \mathbb{Z}_3^n$, where  $(A,\vec{b})$ represents the transformation $\vec{v}\mapsto A\vec{v}+\vec{b}$.  Alternatively, a transformation permutes the points of $AG(n,3)$ by mapping  $n+1$ affinely independent set of points to any $n+1$ affinely independent points.  
The applet Swingset, developed by Coleman, Hartshorn, Long and Mills \cite{CH} provides a nice way to visualize these affine transformations using the card game $\set$ \cite{Set}.  Caps are invariant under the action of $\mathit{Aff}(n,3)$.

The maximal caps in the affine geometry $AG(4,3)$ were first enumerated in 1970, in a paper (written in Italian) by G. Pellegrino \cite{P}.  In 1983, R. Hill \cite{H} proved that all maximal caps are affinely equivalent. Aided by the visualization provided by $\set$,  a rich geometric structure to these caps has been discovered.  A. Forbes \cite{F} found that the 81 points in $AG(4,3)$ can be partitioned into 4 mutually disjoint maximal caps with a single point $\vec{a}$ left.  G. Gordon \cite{Gar} realized that any pair of points that make a line with $\vec{a}$ (there are 40 such pairs) lie in one of the caps in the partition.  
It is the goal of this paper to explore more of the structure of these partitions.  

It is well-understood that the symmetry group of $AG(n,3)$ acts transitively on maximal caps; here, we ask whether that action is 2-transitive on disjoint caps.  Further, does the symmetry group act transitively on partitions of the affine geometry into maximal caps?  We will look at partitions of $AG(n,3)$, for $n=2,3$ and 4; we show in Section \ref{S:2and3} that  the action of the symmetry group is transitive on partitions for $AG(2,3)$ and $AG(3,3)$.  In Section \ref{S:dim4}, we show  that the partitions in $AG(4,3)$ are in two affine equivalence classes in $AG(4,3)$ and isolate the fundamental difference between those classes.  


Finally, $\mathit{Aff}(4,3)$ is of order 1,965,150,720.  In Section \ref{S:groups}, we briefly examine various subgroups of this group that fix particular caps and partitions as sets.


\section{Caps in $AG(n,3)$, with a focus on $n < 4$}\label{S:2and3}

Table \ref{T:capsizes} enumerates the known sizes of maximal caps in $AG(n,3)$ for $n \leq 6$, the only sizes known at this time.  

\begin{table}[h]
\begin{center}
\begin{tabular}{|c|c|c|c|c|c|} \hline
AG$(1,3)$ & AG$(2,3)$ & AG$(3,3)$ & AG$(4,3)$ & AG$(5,3)$ & AG$(6,3)$  \\ \hline
2 & 4 & 9 & 20 & 45 & 112 \\ \hline
\end{tabular}
\end{center}
\caption{All known sizes of maximal caps in $AG(n,3)$}
\label{T:capsizes}
\end{table}%

The sizes for maximal caps in dimensions 1 through 3 can be found by inspection.  
In 1970, Pellegrino provided the first proof that there are 20 points in a maximal cap in $AG(4,3)$ \cite{P}.  Edel, Ferret, Landjev and Storme first  classified the maximal caps in $AG(5,3)$ in 2002 \cite{dim5}.  The results in dimensions 4 and 5 came from looking  at  caps in the projective space $PG(n,3)$, and removing points.   In 2008, A. Potechin found the size of the maximal caps in $AG(6,3)$ \cite{dim6}; these results came from looking at caps in $AG(5,3)$ and analyzing how those can extend to the higher dimension. It is still not known how large maximal caps are in dimensions larger than 6.  In all dimensions where the sizes of maximal caps are known, all maximal caps are affinely equivalent.  Hill first showed this for $AG(4,3)$  \cite{H} in 1983; this result has been extended to all dimensions where the sizes of maximal caps are known in the papers that first identified the maximal caps.

The points in $AG(n,3)$ can be realized as $n$-tuples of elements of $\mathbb{F}_3$.  In this case, as  Davis and Maclagan \cite{DM} point out, lines are easy to identify: Three points in $AG(n,3)$ are collinear if and only if their sum is $\vec{0} \bmod 3$.
%

In dimensions 2 and 4, as we will see, the maximal caps consist of pairs of points from a pencil of lines through a fixed point; we call that point the {\em anchor point.}  In dimensions 3 and 5, maximal caps sum to $\vec{0} \bmod 3$.  In dimension 3 and lower, we can find results about caps simply by inspection.  In this paper, we extend this direct analysis to dimension 4.    We begin in dimension 2.  

To aid in the visualization, we will use the same scheme as is used by Davis and Maclagan \cite{DM}: $AG(2,3)$ is represented by a $3\times3$ grid, as pictured below in Figure \ref{F:AG23}.   There are 12 lines in $AG(2,3)$: 3 horizontal, 3 vertical, 3 diagonals as pictured, and the 3 diagonals in the opposite direction.

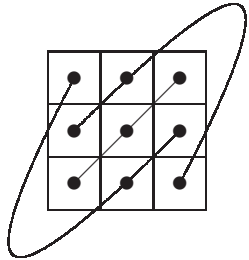
\begin{figure}[h]
\begin{picture}(100,80)(60,0)		%
\multiput(80,10)(0,20){4}{\line(1,0){60}}	
\multiput(80,10)(20,0){4}{\line(0,1){60}}	
\multiput (90,20)(20,0){3}{\circle* {5}}	
\multiput (90,40)(20,0){3}{\circle* {5}}	
\multiput (90,60)(20,0){3}{\circle* {5}}	
\qbezier[1000](90,40)(195,145)(130,20)
\qbezier[1000](130,40)(25,-65)(90,60)
\put(90,20){\line(1,1){40}}
\end{picture}
\begin{center}
\caption{$AG(2,3)$ with one set of diagonal lines shown.}
\label{F:AG23}
\end{center}
\vspace{-.15in}
\end{figure}

\begin{Prop} In $AG(2,3)$, a maximal cap has 4 points and consists of two lines through an anchor point, with the anchor removed; all maximal caps are affinely equivalent.   $AG(2,3)$ can be partitioned into two disjoint caps together with their common anchor point.  Any two partitions are affinely equivalent.
\end{Prop}

\begin{figure}[htbp]
\begin{picture}(80,65)(80,0)		%
\multiput(80,10)(0,20){4}{\line(1,0){60}}	
\multiput(80,10)(20,0){4}{\line(0,1){60}}	
\put(90,60){\circle*{2}}				
\multiput (110,60)(20,0){2}{\circle* {10}}	
\multiput (110,40)(20,-20){2}{\circle {10}}	
\multiput (90,20)(0,20){2}{\circle* {5}}	
\multiput (130,40)(-20,-20){2}{\circle {5}}		
\end{picture}
\vspace{-.2in}
\begin{center}
\caption{A partition of $AG(2,3)$ into 4 caps with anchor point in the upper left.}
\label{F:Dim2ptn}
\end{center}
\vspace{-.2in}
\end{figure}
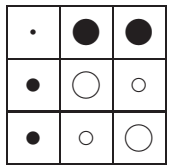 

\begin{proof}
The structure of the maximal caps and their affine equivalence can easily be determined by inspection.  The 4 lines through the point  $\vec{0}$ (in the upper left) are shown  in Figure \ref{F:Dim2ptn} as a pair of points of the same size and shading.   Any two of these pairs will form a maximal cap.  The anchor is uniquely determined by the cap, since the sum of the coordinates for the points in a cap gives the coordinates for the anchor.  Thus, the remaining 4 points must be a cap with the same anchor.  An affine transformation is determined by the images of the anchor,  one large black point and one small black point, so any two caps are equivalent; since the partition is determined by a cap, so all partitions are affinely equivalent as well.
\end{proof}

\medskip
We will represent $AG(3,3)$  by three $3\times3$ grids. Two examples of lines are shown in Figure \ref{F:AG33}.  A line will either consist of three points in one subgrid of $AG(3,3)$ in the same position as a line from $AG(2,3)$,  or three points with one in each of the three subgrids such that, if you superimpose the three subgrids, the points are either in the same position (the open circles in Figure \ref{F:AG33}) or are in the position as a line in $AG(2,3)$ (the solid dots in Figure \ref{F:AG33}).

\begin{figure}[h]
\begin{picture}(125,50)(67.5,7.5)		%
\multiput(60,7.5)(0,15){4}{\line(1,0){45}}	
\multiput(60,7.5)(15,0){4}{\line(0,1){45}}	
\multiput(107,7.5)(0,15){4}{\line(1,0){45}}	
\multiput(107,7.5)(15,0){4}{\line(0,1){45}}	
\multiput(154,7.5)(0,15){4}{\line(1,0){45}}	
\multiput(154,7.5)(15,0){4}{\line(0,1){45}}	
\put (67.5,30){\circle* {6}}
\put (129.5,45){\circle* {6}}
\put (191.5,15){\circle* {6}}
\put (97.5,45){\circle {10}}
\put (144.5,45){\circle {10}}
\put (191.5,45){\circle {10}}
\end{picture}
\vspace{-.1in}
\begin{center}
\caption{$AG(3,3)$ with two sets of collinear points shown.}
\label{F:AG33}
\end{center}
\end{figure}
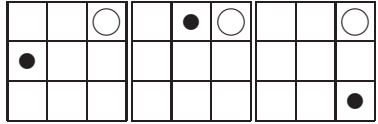

In $AG(3,3)$, the caps sum to $\vec{0}$, so there is no anchor.  $AG(3,3)$ can be partitioned into 3 disjoint maximal caps.  
It is interesting to note that all maximal caps in $AG(5,3)$ have 45 points, which also sum to $\vec{0}$; could it always be true that caps sum to $\vec{0} \bmod 3$ in odd dimensions?  However, $AG(5,3)$ cannot have a similar decomposition into disjoint caps, as 45 does not divide 243.

\begin{Prop} $(1)$ In $AG(3,3)$, all maximal caps are affinely equivalent; the coordinates for a maximal cap sum to $\vec{0} \bmod 3$.  

$(2)$  $AG(3,3)$ can be partitioned into three mutually disjoint maximal caps.  Every maximal cap in $AG(3,3)$ is in a unique partition of $AG(3,3)$; thus, all partitions are equivalent.
\end{Prop}

\begin{proof}
$(1)$ An example of a maximal cap in $AG(3,3)$ is pictured in Figure \ref{F:9-cap}; label these points  $\vec{b_1},\dots,\vec{b_9}$.  The reader can verify that there are no lines in the set of points and that the sum of coordinates is $\vec{0} \bmod 3$.  We were unable to find a proof in the literature that all maximal caps in $AG(3,3)$ are affinely equivalent, so we include a proof here for completeness. 

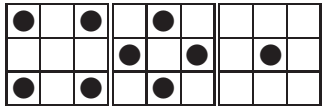
\begin{figure}[h]
\begin{picture}(100,30)(67.5,15) \setlength{\unitlength}{.03cm}		%
\multiput(59.8,7.5)(0,15){4}{\line(1,0){45}}	
\multiput(59.8,7.5)(15,0){4}{\line(0,1){45}}	
\multiput(107,7.5)(0,15){4}{\line(1,0){45}}	
\multiput(107,7.5)(15,0){4}{\line(0,1){45}}	
\multiput(154,7.5)(0,15){4}{\line(1,0){45}}	
\multiput(154,7.5)(15,0){4}{\line(0,1){45}}	
\multiput (67.5,15)(0,30){2}{\circle* {9}}
\multiput (97.5,15)(0,30){2}{\circle* {9}}
\multiput (129.5,15)(0,30){2}{\circle* {9}}
\multiput (114.5,30)(30,0){2}{\circle* {9}}
\put (177,30){\circle* {9}}
\end{picture}
\begin{center}
\caption{A maximal 9-cap in $AG(3,3)$.}
\label{F:9-cap}
\end{center}
\vspace{-.1in}
\end{figure}

First, any point in $AG(3,3)$ that is not in the cap must complete a line from points in the cap, but in fact, every such point completes exactly two distinct lines.  To show this, assume that a point $c$ completes 3 lines.  We can find an affine transformation that sends $c$ to the open dot as shown in Figure \ref{F:pf} and the three pairs of points in a line with $c$ to the large dark points as shown in that figure.  We can now add at most two points to the second subgrid without completing a line; two such points are shown as small dark circles. The cap is now complete and contains only 8 points, a contradiction.  

Thus, each point completes at most 2 lines.  Since there are 18 points not in the maximal cap, and 36 pairs of points in the cap, each point not in the cap completes exactly 2 lines.  

\begin{figure}[h]
\begin{picture}(125,20)(35,18)  \setlength{\unitlength}{.025cm}		%
\multiput(59.8,7.5)(0,15){4}{\line(1,0){45}}	
\multiput(59.8,7.5)(15,0){4}{\line(0,1){45}}	
\multiput(107,7.5)(0,15){4}{\line(1,0){45}}	
\multiput(107,7.5)(15,0){4}{\line(0,1){45}}	
\multiput(154,7.5)(0,15){4}{\line(1,0){45}}	
\multiput(154,7.5)(15,0){4}{\line(0,1){45}}	
\put (67.5,45){\circle{9}}
\multiput (67.5,15)(0,15){2}{\circle* {9}}
\multiput (82.5,45)(15,0){2}{\circle* {9}}
\multiput (114.5,45)(47,0){2}{\circle* {9}}
\multiput (130,15)(0,15){2}{\circle* {5}}
\end{picture}
\begin{center}
\caption{The solid points comprise a complete, non-maximal cap in $AG(3,3)$.}
\label{F:pf}
\end{center}
\vspace{-.1in}
\end{figure}
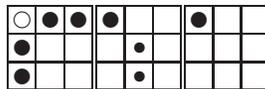

We now show that any maximal cap $C=\{\vec{a_1},\dots,\vec{a_9}\}$  in $AG(3,3)$ is affinely equivalent to the cap in Figure  \ref{F:9-cap}.  We can determine an affine transformation by the images of 4 affinely independent points.  Choose a point $c_1$ not in the cap, and map it to the point in the center of the left-most grid in Figure \ref{F:9-cap}.  Next, map $a_1$ to the point in the upper left corner of that grid.  The image of the point in $C$ that makes a line with $c_1$ and $a_1$, without loss of generality $a_2$, is now mapped to the point in the lower right corner of the same grid.  Once $a_3$ is mapped to the point in the upper right of the left-most grid, the images of all points that are affinely dependent on $c_1$, $a_1$ and $a_3$ are now determined; the only other point in $C$ that is mapped to a point in the left grid is the point that makes a line with $c$ and $a_3$, without loss of generality $a_4$.  Finally, let $c_2$ be any  point not in $C$ and not affinely dependent on $c_1$, $a_1$ and $a_3$, and map that point to the center of the center grid.  The images of all remaining points in $AG(3,3)$ are determined.  There are four points in $C$ that complete two lines with $c_2$; since no other points in $C$ can have images in the left grid, those points must be in the center grid.  It is straightforward to verify that the only points in the center grid that can be images of the points in $C$ are the ones pictured and that this requires the single point in the right grid to be the image of the last point in $C$.  

Now, since any two maximal caps in $AG(3,3)$ are affinely equivalent and summing to zero will be preserved by affine transformations,  the sum of the points $\vec{a_1},\dots,\vec{a_9}$ is also $\vec{0}$ mod 3.

$(2)$ Figure \ref{F:3,3-decomp} shows a decomposition of $AG(3,3)$ into 3 disjoint maximal caps.  Notice that one of the caps is $\vec{b_1},\dots,\vec{b_9}$.  Since all 9-caps of $AG(3,3)$ are affinely equivalent, it suffices to show that this partition is the only partition containing $\vec{b_1},\dots,\vec{b_9}$. 

\begin{figure}[h]
\begin{picture}(125,30)(67.5,20)		%
\multiput(59.8,7.5)(0,15){4}{\line(1,0){45}}	
\multiput(59.8,7.5)(15,0){4}{\line(0,1){45}}	
\multiput(107,7.5)(0,15){4}{\line(1,0){45}}	
\multiput(107,7.5)(15,0){4}{\line(0,1){45}}	
\multiput(154,7.5)(0,15){4}{\line(1,0){45}}	
\multiput(154,7.5)(15,0){4}{\line(0,1){45}}	
\multiput (67.5,15)(0,30){2}{\circle* {9}}
\multiput (97.5,15)(0,30){2}{\circle* {9}}
\multiput (129.5,15)(0,30){2}{\circle* {9}}
\multiput (114.5,30)(30,0){2}{\circle* {9}}
\put (176.5,30){\circle* {9}}
\multiput (82.5,15)(0,30){2}{\circle* {4}}
\multiput (161.5,15)(0,30){2}{\circle* {4}}
\multiput (191.5,15)(0,30){2}{\circle* {4}}
\multiput (68.5,30)(30,0){2}{\circle* {4}}
\put (129.5,30){\circle* {4}}
\multiput (114.5,15)(0,30){2}{\circle {7}}
\multiput (144.5,15)(0,30){2}{\circle {7}}
\multiput (176.5,15)(0,30){2}{\circle {7}}
\multiput (161.5,30)(30,0){2}{\circle {7}}
\put (82.5,30){\circle{7}}
\end{picture}
\begin{center}
\caption{$AG(3,3)$ partitioned into 3 disjoint maximal caps.}
\label{F:3,3-decomp}
\end{center}
\vspace{-.2in}
\end{figure}
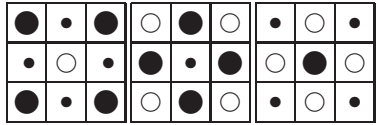

Given any maximal cap $C$ and and 3 parallel planes of $AG(3,3)$,  $C$ must intersect those 3 planes in sets of size 4, 4 and 1 or 3, 3 and 3.  (The only other possibility would be for $C$ to intersect those planes in 4, 3 and 2 points, since a cap cannot have more than 4 points in a plane.  However, we can find an affine transformation so that  each plane corresponds to one coordinate of the vectors; if the cap intersects those 3 planes in 4, 3 and 2 points, the sum of the cap for that coordinate cannot be 0 mod 3.)  Thus, given $C = \{\vec{b_1},\dots,\vec{b_9}\}$, for any partition containing $C$, the other two caps must intersect the three planes in 1 or 4 points.  In the first two planes, if another cap has a 4-point intersection with the plane, then viewing the plane as $AG(2,3)$, the anchor for that cap must be the same as the anchor for the 4 points of $C$, so the only possibility for the first 2 planes is what is shown.  Thus, the other two caps comprising the partition are completely determined.
\end{proof}


\section{Disjoint and intersecting maximal caps of $AG(4,3)$}\label{S:dim4}

We will represent $AG(4,3)$  by a $9\times9$ grid, which we can view as three copies of $AG(3,3)$ or nine copies of $AG(2,3)$ (arranged as $AG(2,3)$).  A line will consist of three points that appear either in the same $3\times3$ subgrid (as a line in $AG(2,3)$), or in three subgrids that correspond to a line in $AG(2,3)$ so that, when the subgrids are superimposed, the points are either in the same position or they are a line in $AG(2,3)$. When coordinatizing $AG(4,3)$, we can have the first two coordinates give the $AG(2,3)$ subgrid, and the second two give the point within that subgrid.  As indicated in Table \ref{T:capsizes}, a maximal cap in $AG(4,3)$ contains 20 points in 10 pairs, where each pair completes a line with the anchor point;  one such cap $S$ is shown in Figure \ref{F:cap}. All maximal caps are affinely equivalent (Pellegrino  \cite{P} and Hill  \cite{H}).  Considering the points as 4-tuples, $S$ is the  first cap lexicographically with $\vec{0}$ as its anchor point.

\begin{figure}[htbp]
\begin{center}
\includegraphics[width=1.3in]{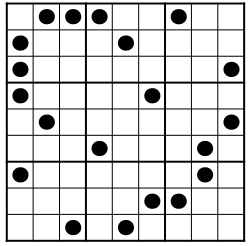}
\caption{A maximal cap $S$ in $AG(4,3)$; the anchor point is in the upper left.}
\label{F:cap}
\end{center}
\vspace{-.1in}
\end{figure} 

\begin{Lem}
For any maximal cap $S$ in $AG(4,3)$, there exists an anchor point $\vec{a}$,  so that the cap consists of 10 pairs of points,  each of which forms a line with $\vec{a}$.  The sum of the coordinates of the points in $S$ is $-\vec{a}  \bmod 3$, so the anchor point is unique.
\end{Lem}

\begin{proof}
One can verify that the set $S$ pictured in Figure \ref{F:cap} contains no lines, so it must be a maximal cap.   $S$ consists of 10 pairs of points, where the third point completing the line for each pair is the point in the upper left, $\vec{0}$.  Since any other maximal cap is affinely equivalent to $S$, the same must be true for all caps.

Further, suppose $S_1$ is an arbitrary maximal cap with an anchor point $\vec{a}$.  Since the coordinates of three collinear points sum to $\vec{0}$ mod 3, if the coordinates for the points in $S_1$ are summed with 10$\vec{a}$, the result must be $\vec{0}$ mod 3.  Thus, the sum of the points in $S_1$ is $-\vec{a} \bmod 3$.  
\end{proof}

The following lemma was originally verified by Forbes via a computer search; we give a direct proof.  Note that  an affine transformation fixing the point corresponding to $\vec{0}$ is  a linear transformation.  This will simplify some of our arguments.

\begin{Lem}
There are 8424 maximal caps with anchor $\vec{0}$; they are all linearly equivalent.
\end{Lem}

\begin{proof}
All maximal caps in $AG(4,3)$ are affinely equivalent. The elements of $GL(4,3)$ send caps with anchor $\vec{0}$ to caps with anchor $\vec{0}$. Thus, by the Orbit-Stabilizer Theorem, we can count the number of caps with anchor $\vec{0}$ by counting the matrices in $GL(4,3)$ that send the cap $S$ to itself.  Since a linear transformation is determined by its action on a basis, we will find such a basis among the vectors corresponding to points in $S$. If we order the points of $S$ (as pictured in Figure \ref{F:cap}) lexicographically as $c_1, -c_1, c_2, -c_2, \cdots, c_{10}, -c_{10}$, we can see that $c_1, c_2, c_3$ and $c_5$ are linearly independent. So we will determine a matrix in $GL(4,3)$ fixing $S$ by specifying the images of those vectors.  

Looking at Figure \ref{F:cap},  $c_1$ can be sent to any of the 20 points.  $c_2$ can be sent to any of the 18 points that don't include the image of $c_1$ and $-c_1$.  The image of $c_3$ is restricted by the fact that once it is chosen, all points from $S$ in the hyperplane determined by $c_1$, $c_2$ and $c_3$ must go to points in $S$ as well, since all points in that hyperplane are linear combinations of  $c_1$, $c_2$ and $c_3$.  Not all choices for the image of $c_3$ will work.  Similarly, the image of $c_5$, the  last point in $S$ not in that hyperplane, does not have full freedom.  The reader can verify that, once the image of $c_1$ and $c_2$ are chosen, there are only 8 possibilities for the images of $c_3$ and $c_5$.  Thus, there are $20\cdot18\cdot8 = 2880$ matrices that fix $S$ as a cap.  Thus, there are $|GL(4,3)|/2880 = 8424$ caps with anchor $\vec{0}$. 
 \end{proof}
 
 The next theorem shows that $AG(4,3)$ can be partitioned into 4 disjoint maximal caps together with their common anchor, just as $AG(2,3)$ was.  This fact was first noticed by Forbes \cite{F} and Gordon \cite{Gar}. 

\begin{Thm}\label{Th:Ptn}
$AG(4,3)$ can be partitioned into 4 mutually disjoint maximal caps together with their common anchor $\vec{a}$.
\end{Thm}

\begin{proof}  One such partition, where $S$ pictured above is one of the maximal caps, is shown in Figure \ref{F:decomp}.  The reader can verify that the claims in the theorem hold for this partition.  
\end{proof}

\begin{figure}[htbp]
\begin{center}
\includegraphics[width=1.3in]{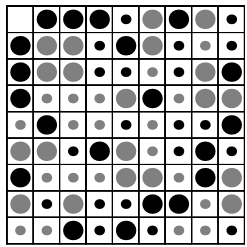}
\caption{A partition of $AG(4,3)$ into 4 disjoint maximal caps; the anchor point $\vec{0}$ is in the upper left.}
\label{F:decomp}
\end{center}
\vspace{-.1in}
\end{figure}

The goal of the rest of this paper is to study these partitions.

\medskip

The next proposition has been verified by computer search.  It would be instructive to have a geometric proof of this fact, as it implies something important about the structure of maximal caps.  The proposition is very useful in understanding the structure of the partitions, for it shows that the partitions in Theorem \ref{Th:Ptn} are the only kind of partitions of $AG(4,3)$ that can include disjoint maximal caps.

\begin{Prop}
Any two maximal caps with different anchor points intersect in at least one point.
\end{Prop}

\begin{proof}
Because any two maximal caps are affinely equivalent, it suffices to verify that a given maximal cap has nonempty intersection with all caps with all other anchor points.  Let $S$ be the maximal cap with anchor $\vec{0}$ pictured  in Figure \ref{F:cap}.  Let $\{S_1,\dots,S_{8424}\}$ be the set of 8424 maximal caps with anchor $\vec{0}$.  For a cap $S_i$ in that set, if we add $\vec{a}$ (mod 3) to each point in $S_i$ (which we write as $S_i + \vec{a}$), we get a cap with anchor $\vec{a}$.  (This is because $S_i + \vec{a}$ must contain no lines, 
and if $S+\vec{a}=T+\vec{a}$ as sets, then $S=T$.)  Thus,  $\{S_i + \vec{a}\}$ is the set of 8424 maximal caps with anchor $\vec{a}$.  

A computer check  ran through all 80 possible anchor points $\vec{a}$ and verified that $S$ and $S_i + \vec{a}$ had nonempty intersection for $1\leq i \leq 8424$.
\end{proof}

The same computer check verified the first claim in the next proposition.  The last claims in the proposition were shown by a different computer search by Forbes \cite{F}.  

\begin{Prop}
Let $S$ be a maximal cap with anchor $\vec{0}$.
There are 198 maximal caps (necessarily with anchor $\vec{0}$) disjoint from $S$.  There are 216 different partitions of $AG(4,3)$ containing $S$ as a block; each of the 198 caps disjoint from $S$ is in at least one of the 216 partitions.  
\end{Prop}

While the group  $\mathit{Aff}(n,3)$ acts transitively on maximal caps, there are three equivalence classes for pairs $(S_1,S_2)$ of disjoint caps.  
Consider Figures \ref{F:1-2completables} and  \ref{F:6completables} below.  Let $S$ be the maximal cap with anchor $\vec{0}$ in large black dots (this is the same cap pictured in Figure \ref{F:cap}).  Three different caps $C$ disjoint from $S$ are pictured in large grey dots in Figures \ref{F:1-2completables}(a), (b) and  \ref{F:6completables}.  In each case, there are 40 points not in $\{\vec{0}\} \cup S \cup C$.  In Figure \ref{F:1-2completables}(a), those points can be partitioned into 2 disjoint maximal caps in only one way; in Figure \ref{F:1-2completables}(b), they can be partitioned into 2 disjoint maximal caps in two different ways.  In Figure \ref{F:6completables}, they can be partitioned into 2 disjoint maximal caps in six different ways.

\begin{figure}[h]
\begin{center}
$\begin{array}{c@{\hspace{1.3in}}c@{\hspace{.3in}}c}
\multicolumn{1}{l}{\mbox{ }} &
	\multicolumn{1}{l}{\mbox{ }}&
	\multicolumn{1}{l}{\mbox{ }} \\ [-0.53cm]
\includegraphics[width=1.25in]{1-comp} &
\includegraphics[width=1.25in]{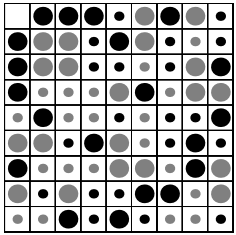} &
\includegraphics[width=1.25in]{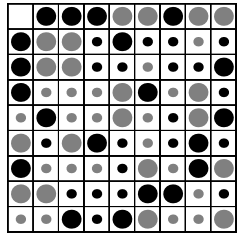}
\end{array}$
\newline \hspace*{-1.1in}(a) 1-completable  \hspace{2.1in} (b)  2-completable 
\caption{Partitions of $AG(4,3)$ containing $S$ (in large black dots) and (a) a 1-completable cap and (b) a 2-completable cap (in large gray dots).}
\label{F:1-2completables}
\end{center}
\end{figure}

\begin{figure}[h]
\begin{center}
$\begin{array}{c@{\hspace{.1in}}c@{\hspace{.1in}}c@{\hspace{.1in}}c@{\hspace{.1in}}c@{\hspace{.1in}}c}
\multicolumn{1}{l}{\mbox{ }} & \multicolumn{1}{l}{\mbox{ }}&
	\multicolumn{1}{l}{\mbox{ }} \\ [-0.53cm]
\includegraphics[width=.9in]{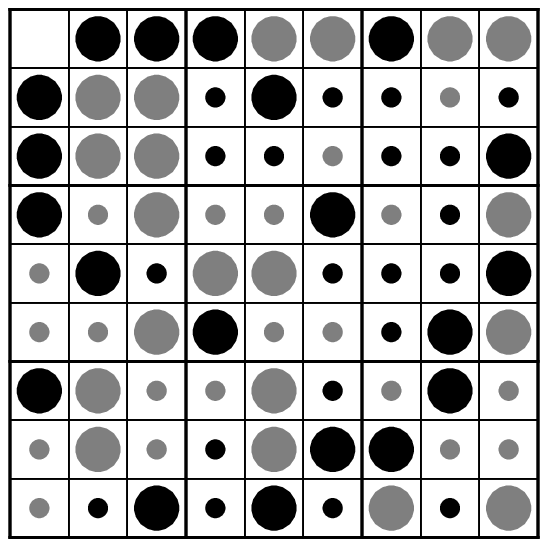} &
\includegraphics[width=.9in]{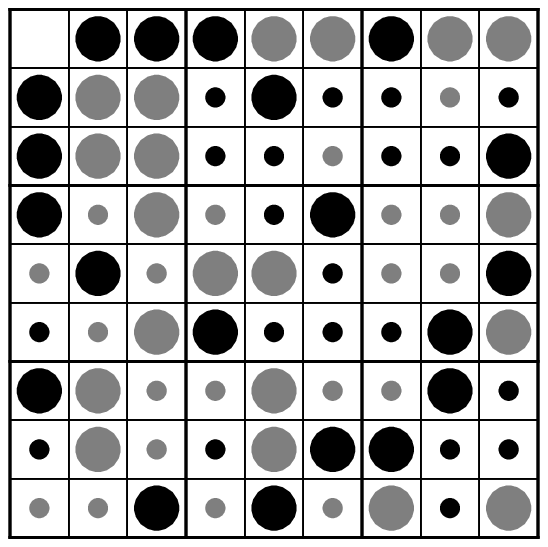} &
\includegraphics[width=.9in]{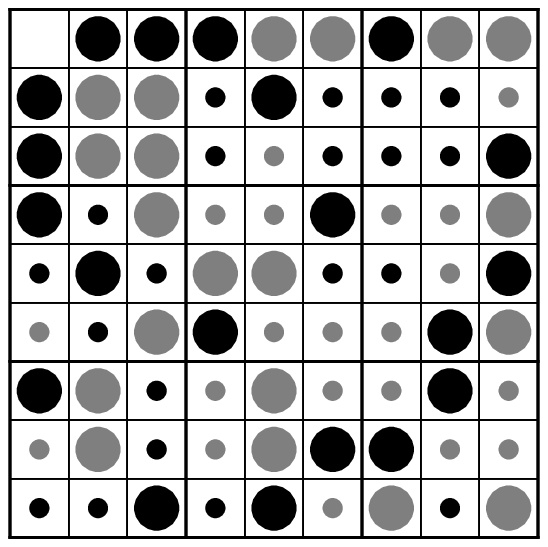} &
\includegraphics[width=.9in]{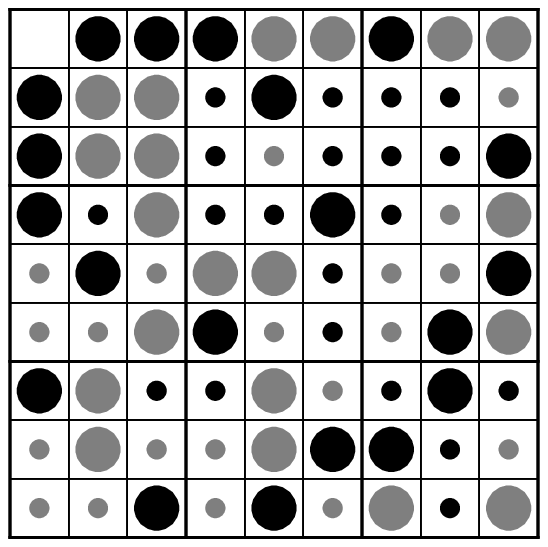} &
\includegraphics[width=.9in]{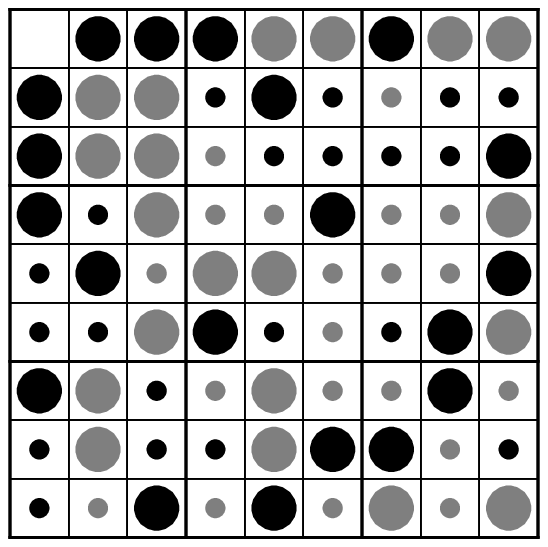} &
\includegraphics[width=.9in]{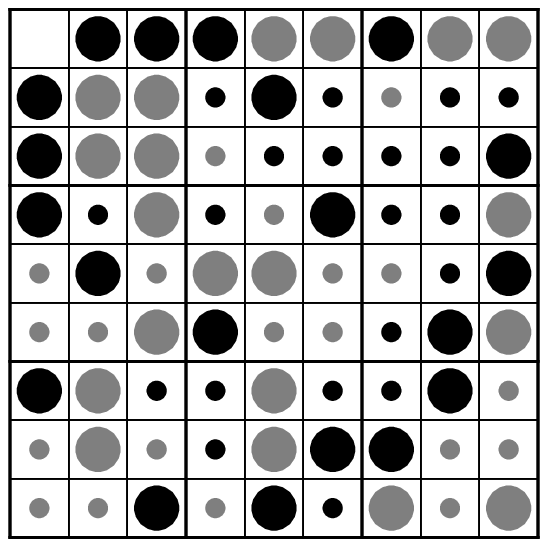}
\end{array}$
\caption{Partitions of $AG(4,3)$ containing $S$ (in large black dots) and a 6-completable cap (in large gray dots).}
\label{F:6completables}
\end{center}
\end{figure}

\begin{Def}
Let $S$ be a maximal cap with anchor $\vec{0}$.  There are 36 maximal caps disjoint from $S$ that appear in only one partition of $AG(4,3)$ containing $S$; there are 90 maximal caps disjoint from $S$ that appear in exactly two partitions of $AG(4,3)$ containing $S$; and  there are 72 maximal caps disjoint from $S$ that appear in exactly six partitions of $AG(4,3)$ containing $S$.  We call these caps {\em $S$-1-completable caps}, {\em $S$-2-completable caps} and {\em $S$-6-completable caps}, respectively.
\end{Def}

Thus, if a maximal cap $S$ is chosen,  the 198 caps disjoint from $S$ are not all affinely equivalent when $S$ is fixed as a set.  
The next proposition summarizes how a linear transformation that fixes $S$ as a set permutes the maximal caps disjoint from $S$ and the partitions containing $S$.  These results were verified by applying linear transformations to caps and partitions.

\begin{Prop}\label{P:transfptns}
Let $S$ be a maximal cap with anchor $\vec{0}$, and let  $\mathcal{T}$ be the group of linear transformations of $AG(4,3)$ that fix  $S$ as a set; let $T \in \mathcal{T}$.  
\begin{enumerate}\renewcommand{\labelenumi}{(\alph{enumi})}
\item If $S_i$ is $S$-$i$-completable, then so is $T(S_i)$, for $i=1,2,6$.  
\item Any partition of $AG(4,3)$ containing $S$ will have two $S$-6-completable caps and either  an $S$-1-completable or an $S$-2-completable cap.  
\item The 216 partitions of $AG(4,3)$ containing $S$ are in two different  equivalence classes under the action of $\mathcal{T}$.  $E_1$ contains 36 partitions that consist of $\{\{\vec{0}\},S, S_1,S_{61},S'_{61}\}$, where $S_1$ is $S$-1-completable and $S_{61}$ and $S_{61}'$ are both $S$-6-completable.  $E_2$ contains 180 partitions that consist of $\{\{\vec{0}\},S, S_2,S_{62},S_{62}'\}$, where $S_2$ is $S$-2-completable and $S'_{62}$ and $S_{62}'$ are both $S$-6-completable.  
\item $\mathcal{T}$ acts transitively on $E_1$ and acts transitively on $E_2$.  If $\,\Pi=\{\{\vec{0}\},S,A,B,C\}$ and $\Pi'=\{\{\vec{0}\},S,A',$ $B', C'\}$, where $A$ and $A'$ are both either $S$-1-completable  or $S$-2-completable and $B$, $C$, $B'$ and $C'$ are $S$-6-completable, then half the matrices in $\mathcal{T}$ that fix $S$ and send $A$ to $A'$  send $B$ to $B'$ and $C$ to $C'$ and half send $B$ to $C'$ and $C$ to $B'$.
\item An $S$-6-completable cap appears in exactly one partition in $E_1$ and in five partitions in $E_2$.
\end{enumerate}
\end{Prop}

We have been considering partitions containing a particular maximal cap $S$ with anchor $\vec{0}$.  Because all maximal caps are affinely equivalent, this was sufficient (and much more convenient) for analyzing the group action.  We now broaden our perspective to consider all partitions with anchor $\vec{0}$, which will extend  to all partitions. 

How many  partitions are there with anchor $\vec{0}$?  These partitions are not all linearly equivalent, but how many equivalence classes are there?  
There are 8424 caps we could have chosen as our fixed cap and 216 partitions containing that cap, but then each partition was counted 4 times.  Thus, there are 454,896 partitions with anchor $\vec{0}$.  These partitions are acted on by the full general linear group, $GL(4,3)$, which has order 24,261,120.  However,  454,896 does not divide 24,261,120, so the partitions must be in at least two equivalence classes.  To understand these equivalence classes, we need to understand how the caps in the partitions behave with respect to each other.

Let $\{\{\vec{0}\},A,B,C,D\}$ be a partition of $AG(4,3)$ into 4 mutually disjoint maximal caps together with their anchor point.  Clearly, if $B$ is $A$-1-completable (respectively $A$-2-completable, $A$-6-completable), then $A$ is $B$-1-completable (respectively $B$-2-completable, $B$-6-completable).  This motivates the next definition.

\begin{Def}
Let $\{\{\vec{0}\},A,B,C,D\}$ be a partition of $AG(4,3)$ into 4 mutually disjoint maximal caps together with their anchor point.  We say $\{A,B\}$ is a {\em 1-completable pair} (respectively a {\em 2-completable pair}) if the set $\{A,B\}$ appears in exactly one partition (respectively exactly two partitions). 
\end{Def}

The next lemma shows that the two 6-completable caps in a partition are themselves a 1-completable or 2-completable pair.  Further, a partition of $AG(4,3)$ into 4 mutually disjoint maximal caps and their anchor point must consist of two pairs of caps, where either both pairs are 1-completable or both are 2-completable.  This also means that the other pair has both caps 6-completable with respect to either cap in the first pair. 

\begin{Lem}\label{L:pairs} Let $\{\{\vec{0}\},A,B,C,D\}$ be a partition of $AG(4,3)$  into 4 mutually disjoint maximal caps together with their anchor point.
\begin{enumerate}
\item If $B$ is $A$-1-completable or $A$-2-completable, then $D$ is $C$-1-completable or $C$-2-completable.
\item Let $\{\{\vec{0}\},A,B,C,D\}$ be a partition of $AG(4,3)$ into 4 mutually disjoint maximal caps together with their anchor point.  Then $\{A,B\}$ is a 1-completable pair if and only if $\{C,D\}$ is a 1-completable pair.  Thus, $\{A,B\}$ is a 2-completable pair if and only if $\{C,D\}$ is a 2-completable pair.
\end{enumerate}
\end{Lem}

\begin{proof}
(1)  If $B$ is $A$-1-completable or $A$-2-completable, then if we consider the partition $\{\{\vec{0}\},B,A,C,D\}$ and think of $B$ as the fixed maximal cap, we have a partition that can contain only one 1- or 2-completable cap with respect to $B$, and since $A$ is 1- or 2-completable with respect to $B$, we must have that $C$ and $D$ are both 6-completable with respect to $B$, so $B$ is also 6-completable with respect to $C$ and $D$.  Since both $A$ and $B$ are 6-completable with respect to $C$ and $D$,  then considering the partition as fixing $C$, it must also be true that $C$ and $D$ are 1- or 2-completable with respect to each other.  

(2) Given the partition $\{\{\vec{0}\},A,B,C,D\}$,  assume that $\{A,B\}$ is a 1-completable pair.  Then $C$ is 6-completable with respect to $A$.  Let $\Pi_i = \{\{\vec{0}\}, A, B_i, C, D_i\}$, $2\leq i \leq6$ be the five additional partitions containing $A$ and $C$.  Then $B$ is the  unique 1-completable with respect to $A$ among those partitions (by Proposition \ref{P:transfptns}(e)), so WLOG,  $\{A,B_2\},\dots, \{A,B_6\}$ are 2-completable pairs.

By Proposition \ref{P:transfptns}(d), there are linear transformations $T_i$ fixing $A$ and sending $\Pi_2$ to $\Pi_i$, $i = 3, \dots, 6$ that fix $C$ as well.  Shifting our point of view so that we are thinking of these partitions as fixing $C$, by Proposition \ref{P:transfptns}(a) and (e), the existence of the transformations $T_i$ imply that  $D_2, \dots, D_6$ must  be 2-completable with respect to $C$, so $\{C,D\}$ is a 1-completable pair.  This means that the pairing of caps in any partition must have either two 1-completable pairs or two 2-completable pairs.  

\end{proof}

We can now put these results together to give the equivalence classes of partitions of $AG(4,3)$ with an arbitrary anchor point.  The affine group acting on the elements of $AG(4,3)$ is $ \mathit{Aff}(4,\mathbb{F}_3) \cong  GL(4,3)   \ltimes AG(4,3)$.  This action sends caps to caps, so it also sends partitions to partitions.  Thus, we can extend the structures we've  found to all possible partitions of $AG(4,3)$.  

\begin{Thm}
The partitions of $AG(4,3)$ into 4 mutually disjoint maximal caps and the associated anchor point $\vec{a}$ are in two equivalence classes under the action of the affine group $\mathit{Aff}(4,3)$.  One equivalence class consists of partitions with two 1-completable pairs, and the other consists of partitions with two 2-completable pairs.  
\end{Thm}

\begin{proof}
From Lemma \ref{L:pairs},  all partitions consist of two 1-completable pairs or two 2-completable pairs.  Extending Theorem  \ref{P:transfptns}(a), if   we see that if $\{\{\vec{0}\},A,B,C,D\}$ and $\{\{\vec{a}\},A',B',C',D'\}$ are two partitions and $\vec{a}\cdot T$ is an element of $\mathit{Aff}(4,3)$, taking $A$ to $A'$, etc.,  then $B$ is 1-completable with respect to $A$ if and only if $B'$ is 1-completable with respect to $A'$.  So, 1-completable pairs must go to 1-completable pairs, and 2-completable pairs must go to 2-completable pairs.  Thus, one equivalence class under the action of $\mathit{Aff}(4,3)$ is the set of partitions containing two 1-completable pairs; the other is the set of partitions containing two 2-completable pairs.
\end{proof}


\section{Subgroups of the affine group acting on partitions}\label{S:groups}

The full automorphism group of $AG(4,3)$ is the group of affine transformations, the affine group $\mathit{Aff}(4,3) = GL(4,3) \ltimes \mathbb{Z}_3^4$.  This group is of order 1,965,150,720.  Let  $\vec{0}=(0,0,0,0)$ in $AG(4,3)$ and consider the stabilizer of $\vec{0}$, $GL(4,3)$, of order 24,261,120.  Since $\mathit{Aff}(4,3)$ is 2-transitive on points in $AG(4,3)$, $GL(4,3)$ is transitive  on points.   $GL(4,3)$ is transitive on caps with anchor $\vec{0}$, but not 2-transitive on caps with anchor $\vec{0}$: while we can send any cap $C$ with anchor $\vec{0}$  to another cap $C'$ with anchor $\vec{0}$,  we can only send $C$-1 completable (respectively $C$-2-completable, $C$-6-completable) caps  to $C'$-1 completable (respectively $C'$-2-completable, $C'$-6-completable) caps, and that action extends to the action of $\mathit{Aff}(4,3)$ on all caps.  Thus,  without loss of generality, we can understand the full group action by  considering the stabilizer $G$ of  one particular maximal cap $S$ as a set (so $G$ also necessarily stabilizes $\vec{0}$), a subgroup of  size 2880.  

The results in this section were found using {\em Mathematica} \cite{Mica} to compute with matrices and  GAP \cite{gap} to analyze the structure of the groups of matrices.  Recall, $E_1$ is the set of partitions containing a particular cap $S$ (with anchor point $\vec{0}$), an $S$-1-completable cap and a second 1-completable pair (where both both caps in that pair are $S$-6-completable); $E_2$ is the set of partitions containing $S$, an $S$-2-completable cap and another 2-completable pair (where both both caps in that pair are $S$-6-completable).

$G$ is transitive on the partitions in $E_1$ and transitive on the partitions in $E_2$.  The subgroup $G_1$ of transformations of determinant 1 is transitive on the partitions in $E_1$ but not transitive on the partitions of $E_2$.  If $\Pi_2$ is a partition in $E_2$, then $\{T(\Pi_2)\}$, $T \in G_1$, is half of the partitions of $E_2$, and each $S$-2-completable cap appears exactly once in that set.  This means that each 2-completable also appears exactly once in $\{T(P_2)\}$, $T \in G -G_1$.

Let $\Pi_1$ be a partition in $E_1$ and let $S_1$ be the $S$-1-completable cap in $\Pi_1$. There is a subgroup $H$  of $G$ of size 40 stabilizing the individual caps of $\Pi_1$ as sets.  These transformations are all of determinant 1. $H$ is nonabelian and has a unique subgroup isomorphic to $\mathbb{Z}_{20}$ and so is isomorphic to $\mathbb{Z}_{20}\rtimes \mathbb{Z}_2$.  There are also $40$ transformations that stabilize $S$ and $S_1$ as sets and switch the two $S$-$6$-completables in the decomposition; these are all of determinant 2.  Thus there is a group of order $80$ stabilizing $S$ and $S_1$ as sets. 

 Let $\Pi_2 $ be a partition in $E_2$ and let $S_2$ be the $S$-2-completable cap in $\Pi_2$.   There is a subgroup $K$  of $G$ of size 8 fixing the individual caps of $\Pi_2$ as sets; these transformations are all of determinant 1. $K$ is isomorphic to $\mathbb{Z}_{4}\times \mathbb{Z}_2$.  There are also 8 transformations that stabilize $S$ and $S_2$ as sets and switch the two $S$-$6$-completables in the decomposition; these are also all of determinant 1. This group of order 16 fixing $\Pi_2$ as a collection of caps is is isomorphic to $\mathbb{Z}_{4}\rtimes \mathbb{Z}_4$. There is another set of $16$ linear transformations that stabilize $S$ and $S_2$, but which send  the two $S$-$6$-completables in $D_2$ to the other 2-completable pair that appears in a partition with $S$ and $S_2$.   These transformations all have determinant 2. The group of order $32$ stabilizing $S$ and $S_2$ is isomorphic to $(\mathbb{Z}_8 \times \mathbb{Z}_2) \rtimes \mathbb{Z}_2$.

Let $S_6$ be  $S$-$6$-completable.  Then exactly one of the partitions containing $S$ and $S_6$ has an $S$-1-completable cap, from Proposition \ref{P:transfptns}(e).   The subgroup of $G$ fixing $S$ and $S_6$ is the same subgroup of order $40$ that fixes $S$ and the unique $S$-1-completable associated with $S_6$.

Finally, $G$ has 144 elements of order 5, so there are $36$ distinct subgroups isomorphic to $\mathbb{Z}_5$.  
 Each of these subgroups fixes a unique element of $E_1$.  Three elements of order 5  generate the subgroup containing all the elements of order 5, which is isomorphic to $A_6$.
 
 \medskip
How these subgroups permute the partitions and the 1- and 2-completable pairs could prove instructive in understanding the geometric structure of the partitions.  


\end{document}